\documentclass[12pt, a4paper]{article}
\usepackage{amsthm}
\usepackage{amsmath}
\usepackage{amssymb}
    \usepackage{bm}
\usepackage{fullpage}
\usepackage{enumerate}
\usepackage{mathrsfs}
\usepackage[all]{xy}
\usepackage[usenames,dvipsnames]{color}
\usepackage{graphicx}
\usepackage{bbm}
\usepackage[utf8]{inputenc}
\usepackage{enumitem}
\usepackage{multicol}
\usepackage[affil-it]{authblk}

\newtheorem{theorem}{Theorem}[section]
\newtheorem{remark}{Remark}[section]
\newtheorem{corollary}{Corollary}[section]
\newtheorem{lemma}{Lemma}[section]
\newtheorem{claim}{Claim}[section]
\newtheorem{proposition}{Proposition}[section]
\newtheorem{example}{Example}[section]

\newtheorem{setting}{Setting}[section]
\newtheorem{lemmaA}{Lemma A.\hspace{-.1cm}}

\newtheorem{definition}{Definition}[section]

\newcommand{\R}{\mathbb{R}}

\newcommand{\N}{\mathbbm{N}}

\newcommand{\E}{\mathbbm{E}}

\newcommand{\Levy}{L\'{e}vy\text{ }}

\newcommand{\dP}{\mathbbm{P}}
\newcommand{\mcl}{\mathcal}
\newcommand{\LargerCdot}{\raisebox{-0.25ex}{\scalebox{1.2}{$\cdot$}}}
\providecommand{\keywords}[1]{\textbf{\textit{Keywords:}} #1}
\newcommand{\mypar}{\par\text{ }}
\newcommand{\Cov}{\mathrm{Cov}}
\newcommand{\ccdot}{\hspace{-.05cm}  \cdot \hspace{-.075cm}}
\newcommand\independent{\protect\mathpalette{\protect\independenT}{\perp}}
\def\independenT#1#2{\mathrel{\rlap{$#1#2$}\mkern2mu{#1#2}}}
\allowdisplaybreaks
\numberwithin{equation}{section}

\title{On the association and other forms of positive dependence for Feller processes}
\author{Eddie Tu%
 \thanks{\underline{Institution}: Dickinson College; \underline{Postal Address}:  Dickinson College, Department of Mathematics and Computer Science, PO Box 1773, Carlisle, PA 17013; \underline{Electronic address}: tue@dickinson.edu}}

\begin{document}

\maketitle

\begin{abstract}
We characterize various forms of positive dependence, such as association, positive supermodular association and dependence, and positive orthant dependence, for jump-Feller processes. Such jump processes can be studied through their state-space dependent \Levy measures. It is through these \Levy measures where we will provide our characterization. 
Finally, we present applications of these results to stochastically monotone Feller processes, including \Levy processes, the Ornstein-Uhlenbeck process, pseudo-Poisson processes, and subordinated Feller processes. 
\end{abstract}

\noindent\keywords{association, supermodular dependence, supermodular association, orthant dependence,  Feller processes, \Levy processes, integro-differential operators,  symbols}

\section{Introduction}

Multi-dimensional Feller processes have been useful for modeling the evolution of dynamical systems that are spatially inhomogeneous. These processes have been important  models in finance and physics \cite{Bottcher2010}. Of a particular interest is the study of the dependence between the marginal processes. Some different notions of positive dependence include association (A), positive supermodular association (PSA), positive supermodular dependence (PSD), and positive orthant dependence (POD). 
If a process exemplifies a certain notion of positive dependence between the marginals, then one can better study the evolution of the process.
\par

It is known that \Levy processes in $\R^d$ can be characterized by their characteristic triplet $(b, \Sigma, \nu)$, where $b\in\R^d$ is the non-random linear drift, $\Sigma$ is the covariance matrix of the (continuous) Brownian motion, and $\nu$ is the \Levy measure which characterizes the jump behavior of the process. Feller processes have behavior that is ``locally-L\'{e}vy," i.e.  for a Feller process $(X_t^x)_{t\geq0}$ that starts at point $x$ ($X_0^x = x$ a.s.), there exists a \Levy process $(Y_t)_{t\geq0}$ such that, in short-time, the behavior of $(X_t^x)_{t\geq0}$ can be approximated by the behavior of $(Y_t + x)_{t\geq0}$ \cite[p.46]{Schilling2013}. This idea is related to the notion that, if the domain $\mcl{D}(\mcl{A})$ is ``rich", i.e. contains $C_c^\infty(\R^d)$, the space of smooth functions with compact support, then the Feller process can be described by a characteristic triplet $(b(x), \Sigma(x), \nu(x,dy))$, where the function $b:\R^d\rightarrow\R^d$ represents non-random component, $\Sigma:\R^d\rightarrow\R^{d\times d}$ represents the continuous diffusion-like behavior, and $x\mapsto\nu(x,dy)$ is a measurable kernel representing the jump behavior of the process. Unlike the \Levy process, the Feller process' triplet has dependence on $x$, the state variable of the process, representing its spatial inhomogeneity. It is these triplets through which we will characterize the different notions of positive dependence.

\par

Association, the strongest form of positive dependence that we will examine, has been well-studied for infinitely divisible distributions. Infinitely divisible random vectors $X$ also have a characteristic triplet $(b, \Sigma, \nu)$ by the famous L\'{e}vy-Khintchine formula, where $b$ represents the non-random component, $\Sigma$ is covariance of the Gaussian component, and $\nu$ is the \Levy measure of the Poissonian component. Pitt (1982) characterized association for Gaussian distributions $(b,Q,0)$ under the condition that the entries $\Sigma_{ij}$ of $\Sigma$ are non-negative \cite{Pitt1982}. Resnick (1988) proved a sufficient condition for association of Poissonian distributions $(0,0,\nu)$ is that $\nu$ be concentrated on the positive and negative orthants $\R_+^d$ and $\R_-^d$ \cite{Resnick1988}, i.e.
\begin{equation}
\label{resnick}
\nu((\R_+^d \cup \R_-^d)^c) = 0.
\end{equation}
These results lead to the characterization of association between the marginal processes of a \Levy process, since, for a \Levy process $Y = (Y_t)_{t\geq0}$, $Y_t$ is infinitely divisible for each $t\geq0$, and the process can be described by its characteristic triplet $(b,\Sigma,\nu)$. Herbst and Pitt (1991) extended Pitt's result in \cite{Pitt1982} to Brownian motion with covariance matrix $\Sigma$ \cite{Herbst1991}. For jump-\Levy processes $Y\sim (0,0,\nu)$, Samorodnitsky (1995) showed that  condition (\ref{resnick}) is a sufficient and necessary condition for the association of each $Y_t$ \cite{Samorodnitsky1995}. This result was also  proven by Houdr\'{e} et. al. (1998) using a covariance identity \cite{Houdre1998}.   B\"{a}uerle et. al. (2008) extended Samorodnitsky's results for jump-\Levy processes to association in time, and showed that condition (\ref{resnick}) is also equivalent to PSD and POD \cite{Bauerle2008}. Liggett (1985) proved a necessary and sufficient condition for association of stochastically monotone Markov processes on compact state spaces based on the generator of the process \cite{Liggett1985}. Szekli (1995) and R\"{u}schendorf (2008) extended this result to more general state spaces \cite[Ch.3.7]{Szekli1995}, \cite[Cor.3.1]{Ruschendorf2008}. R\"{u}schendorf also extended the Liggett condition for PSA of the Markov process \cite[Cor.3.4]{Ruschendorf2008}. 

\par

In this paper, we want to characterize various forms of positive dependence for  stochastically monotone Feller process. Those forms of dependence include association, weak association (WA), PSA, PSD, POD, positive upper orthant dependence (PUOD), and positive lower orthant dependence (PLOD). The association of diffusion processes, i.e. $(b(x), \Sigma(x),0)$, has been characterized by Chen \cite{Chen1993}, so we will only focus on jump-Feller process $(b(x),0,\nu(x,dy))$. Association of jump-Feller processes, i.e. $(b(x), 0, \nu(x,dy))$, was given by Wang (2009) \cite[Thm.1.4]{Wang2009}, but under certain continuity and integrability conditions of the characteristic triplet (see Remark \ref{rem:jmwang}). Here, we will relax those conditions, allowing us to consider a larger class of Feller processes. Additionally, we characterize WA, PSA, PSD,  POD, PUOD, and PLOD for jump-Feller processes. Our techniques extend the ideas of Liggett, Szekli, and R\"{u}schendorf to the extended generator of the process, an integro-differential operator. We use ideas of the probabilistic symbol $p(x,\xi)$ of the process developed by Jacob and Schilling \cite[p.57-58]{Schilling2013}. Furthermore, for proving the necessary condition of association, WA, PSA, PSD, POD, PUOD,  PLOD, we use the technique of small-time asymptotics of the Feller process \cite{KuhnLaplace2016}, which will allow us to surpass the use of the (extended) generator and use solely the state-space dependent \Levy measure $\nu(x,dy)$. Finally, we provide examples of Feller processes satisfying the conditions of our main results.

\par

In a concurrent paper of ours, titled ``Association and other forms of positive dependence for Feller evolution systems" \cite{Tu2018b}, we  characterize dependence structures for Feller evolution processes (FEP), which are time-inhomogeneous Markov processes  having strongly continuous Markov evolutions and L\'{e}vy-type behavior. These FEPs are more general than the Feller processes (time-homogeneous) in this paper, but we need the results of this paper in order to characterize dependence structures of FEPs. We utilize B\"{o}ttcher's transformation of time-inhomogeneous FEPs into time-homogeneous Feller processes (see \cite{Bottcher2013})  and, in a non-trivial way, apply our results in this paper to prove characterizations of positive dependence for FEPs. This yields positive dependence characterizations for interesting time-inhomogeneous processes, like additive processes. For a more comprehensive overview of time-inhomogeneous Markov processes, we recommend the reader explore the paper by R\"{u}schendorf et. al. \cite{Ruschendorf2016}, which also discusses comparison theorems of time-inhomogeneous Markov processes.

\par

The present paper is organized in the following way. In Section \ref{sec:background}, we give some background on the positive dependence structures, association, WA, PSA, PSD,  POD, PUOD, and PLOD, along with definitions of various stochastic orderings. We also provide background on \Levy processes, Feller processes, and the different tools we use to analyze them. In Section \ref{sec:mainresults}, we state and prove our main results about the positive dependence structures of jump-Feller processes. Finally, in Section \ref{sec:examples}, we give a collection of interesting examples of multi-dimensional Feller processes to which we can apply these results.

\section{Background}
\label{sec:background}

\subsection{Dependence and stochastic orderings}

Let $X = (X_1,...,X_d)$ be a random vector in $\R^d$. We say $X$ is \textbf{positively correlated (PC)} if $\Cov(X_i,X_j)\geq0$ for all $i,j\in\{1,...,d\}$. This is one the weakest forms of positive dependence, and we are interested in stronger forms of positive dependence which will be of greater use in our study of stochastic processes. \textit{Association} is the strongest form of positive dependence that we will study. 

\begin{definition}
\label{def:assoc}
{\rm 
$X=(X_1,...,X_d)$ is \textbf{associated (A)} if we have $$\Cov(f(X), g(X)) \geq0,$$
for all $f,g:\R^d\rightarrow\R$ non-decreasing in each component, such that $\Cov(f(X),g(X))$ exists.
}
\end{definition}


We will also study other forms of positive dependence that are weaker than association, but stronger than positive correlation. We list them below.

\begin{definition}
\label{def:WA}
{\rm 
A random vector $X=(X_1,...,X_d)$ is \textbf{weakly associated (WA)} if, for any pair of disjoint subsets $I,J\subseteq\{1,..,d\}$, with $|I| = k$, $|J|=n$,
\begin{equation*}
\label{WA}
\Cov(f(X_I), g(X_J))\geq0,
\end{equation*}
where $X_I := (X_i:i\in I)$, $X_J := (X_j:j\in J)$, for any  $f:\R^k\rightarrow\R$, $g:\R^{n}\rightarrow\R$ non-decreasing, such that  $\Cov(f(X_I), g(X_J))$ exists.
}
\end{definition}


\begin{definition}
\label{def:psa}
{\rm 
$X$ is \textbf{positive supermodular associated (PSA)} if $\Cov(f(X), g(X))\geq0$ for all $f,g\in \mcl{F}_{ism} := \{h:\R^d\rightarrow\R, \text{ non-decreasing, supermodular}\}$. $f$
\textbf{supermodular} means, for all $x,y\in\R^d$, $f(x\wedge y) + f(x\vee y) \geq f(x) + f(y),$
where $x\wedge y$ is the component-wise minimum, and $x\vee y$ is the component-wise maximum. 
}
\end{definition}

Now let $\hat{X} = (\hat{X}_1,...,\hat{X}_d)$ be a random vector such that for all $i$, $\hat{X}_i \stackrel{d}= X_i$ and $\hat{X}_i$'s are mutually independent.

\begin{definition}
\label{def:psd}
{\rm 
$X$ is \textbf{positive supermodular dependent (PSD)} if, for all $f:\R^d\rightarrow\R$ supermodular, 
$\E f(\hat{X}) \leq \E f(X)$.
}
\end{definition}

\begin{definition}
\label{def:puod}
{\rm 
$X$ is \textbf{positive upper orthant dependent (PUOD)} if for all $t_1,...,t_d\in\R$,
\begin{align*}
\dP(X_1 >t_1,...,X_d >t_d) \geq \dP(X_1>t_1)...\dP(X_d>t_d).
\end{align*}
}
\end{definition}

\begin{definition}
\label{def:plod}
{\rm 
$X$ is \textbf{positive lower orthant dependent (PLOD)} if for all $t_1,...,t_d\in\R$,
\begin{align*}
\dP(X_1 \leq t_1,...,X_d \leq t_d) \geq \dP(X_1\leq t_1)...\dP(X_d\leq t_d).
\end{align*}
}
\end{definition}

\begin{definition}
\label{def:pod}
{\rm 
$X$ is \textbf{positive orthant dependent (POD)} if  $X$ is PUOD and PLOD. 

\par

One can also state another equivalent definition to PUOD (PLOD). For $i=1,...,d$, let $f_i:\R\rightarrow\R_+$ be non-decreasing (non-increasing) functions.  Then $X=(X_1,...,X_d)$ \textbf{PUOD} (\textbf{PLOD}) if and only if 
\begin{center}
$\E \left(\prod_{i=1}^d f_i (X_i) \right) \geq \prod_{i=1}^d \E f_i (X_i).$
\end{center}
}
\end{definition}

\noindent
\underline{Note}: Definition \ref{def:assoc} first appeared in Esary et. al. \cite{Esary1967},  Definition \ref{def:WA} in Burton et. al. \cite{Burton1986}, Definition \ref{def:psa} in R\"{u}schendorf \cite[p.284]{Ruschendorf2008}, Definition \ref{def:psd} in Hu \cite{Hu2000}, and Definitions \ref{def:puod}-\ref{def:pod}  in Lehmann \cite{Lehmann1966}. Definitions \ref{def:psd}-\ref{def:pod} can also be stated in terms of stochastic orderings. For more on this, we refer the reader to M\"{u}ller and Stoyan's book \cite[Ch.3]{Mueller2002}. It is useful to see the relationship between these different forms of positive dependence. We state the relationships in the following proposition.

\begin{proposition}
\label{propdepmap}
The  implications in Figure \ref{fig:prop2_1} hold.
\begin{center}
\begin{figure}[h]
  \centering
  \includegraphics[trim = {0cm 0cm 0cm 0cm}, scale=.75]{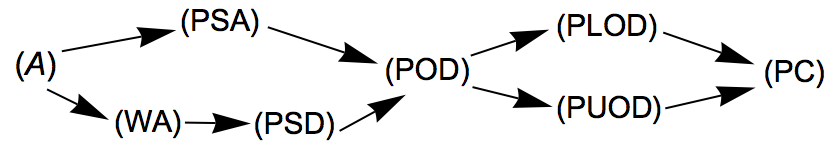}\\
  \caption{Implication map of various positive dependence structures}\label{fig:prop2_1}
\end{figure}
\end{center}
\vspace{-.6cm}



\end{proposition}

\begin{proof}
Proofs for these implications can be found in M\"{u}ller and Stoyan's book \cite[Ch.3]{Mueller2002}, and implications involving PSD can be found in \cite{Christofides2004}.
\end{proof}

These notions of dependence can be extended from random vectors to stochastic processes. Let $X = (X_t)_{t\geq0}$ be a stochastic process in $\R^d$.

\begin{definition}
\label{def:assocspacetime}
{\rm 
(a) Process $X$ is \textbf{associated in space} or \textbf{spatially associated} if, for every $t\geq0$, the random vector $X_t = (X_t^{(1)},...,X_t^{(d)})$ is associated.

\par
  (b) Process $X$ is \textbf{associated in time} or \textbf{temporally associated} if, for all \\$0\leq t_1< ... <t_n$, the random vector $(X_{t_1},...,X_{t_n})$ in $\R^{dn}$ is associated.}
\end{definition}

\begin{remark}
{\rm
\begin{enumerate}[noitemsep, label = (\roman*)]
\item Clearly, (b) is a stronger than (a) in the above definition
\item We can define other forms of positive dependence in stochastic processes if we replace ``associated" in Definitions \ref{def:assocspacetime} (a) and (b) with ``WA," ``PSA," ``PSD," ``POD," ``PUOD," ``PLOD."
\item Definition (a) is equivalent to the statement that the ``process preserves positive correlations," as given in  \cite[p.80]{Liggett1985} and \cite{Chen1993}. 
\end{enumerate}
}
\end{remark}




\subsection{Feller processes, extended generators, small-time asymptotics}




\subsubsection{Feller process}

Consider a time-homogeneous Markov process $X = (X_t)_{t\geq0}$ on the space $(\Omega, \mcl{G}, (\mcl{G}_t)_{t\geq0}, \dP^x)_{x\in\R^d}$ on state space $\R^d$. $(\mcl{G}_t)_{t\geq0}$ is the filtration, and the index ``$x$"  indicates the starting point of the process: $\dP^x(X_0 = x) = 1$. We associate with a Markov process $X$ a positivity-preserving, contraction semigroup of bounded operators $(T_t)_{t\geq0}$ defined by 
\begin{center}
$T_t f(x) := \E^x f(X_t), \hspace{.3cm} x\in\R^d,$
\end{center}
 where $f\in B_b(\R^d)$, the space of bounded measurable functions on $\R^d$.
Let $(C_0(\R^d), ||\cdot||_\infty)$ be the Banach space of continuous functions that vanish at infinity, i.e. $\lim_{|x|\rightarrow\infty}f(x) = 0$, where $||\cdot||_\infty$ is the sup-norm. Define $\mcl{F}_i:=\{f:\R^d \rightarrow\R, \text{ non-decreasing in each component}\}$. A Markov process is \textbf{stochastically monotone} if $T_t f\in\mcl{F}_i$ for all $f\in \mcl{F}_i$.  We define the \textbf{generator} $\mcl{A}$ of the process $X$ to be 
\begin{equation}
\label{generator}
\mcl{A} f := \lim_{t\searrow0}\frac{T_t f - f}{t},
\end{equation}
for all $f\in\mcl{D}(\mcl{A})$, where $\mcl{D}(\mcl{A})$ is the \textbf{domain of the generator} defined to be $$\mcl{D}(\mcl{A}) = \{u\in C_0(\R^d): \text{ limit on RHS of (\ref{generator}) exists uniformly}\}.$$
The Markov process is a \textbf{Feller process} if the semigroup $(T_t)_{t\geq0}$ satisfies the following properties:
\begin{center}
(i) $T_t:C_0(\R^d)\rightarrow C_0(\R^d)$, \hspace{1cm}  (ii)  $\lim_{t\rightarrow0}||T_tu-u||_{\infty}=0$.
\end{center}


If additionally, the domain of the generator contains smooth functions with compact support, i.e. $\mcl{D}(\mcl{A}) \supset C_c^\infty(\R^d)$, we call the process $X$ a \textbf{rich Feller} process. It follows from Courr\`{e}ge's Theorem \cite{Courrege1965} that $-\mcl{A}$ becomes a pseudo-differential operator $p(x,D)$ on the space of $C_c^\infty(\R^d)$: $\mcl{A}|_{C_c^\infty(\R^d)} = -p(x,D) $, where $p(x,D)$ is defined to be 
\begin{equation} 
\label{pseudo} 
\mcl{A} f(x) = -p(x,D) f (x) = (2\pi)^{-d/2}  \int_{\R^d} e^{i \xi\cdot x} p(x,\xi) \hat{f}(\xi) d\xi, \hspace{.3cm} f\in C_c^\infty(\R^d).
\end{equation}
The function $-p(x,\cdot)$ is a continuous negative definite function, in the sense of Schoenberg, for all $x\in\R^d$, which yields a L\'{e}vy-Khintchine representation for each $x$:
\begin{equation}
\label{symbol}
-p(x,\xi) = - i b(x) \cdot \xi + \frac{1}{2} \xi \cdot \Sigma(x) \xi - \int_{\R^d\setminus\{0\}} (e^{i \xi\cdot y} - 1 - i \xi \cdot y\chi(y)) \nu(x,dy),
\end{equation}
\noindent where $\chi:\R^d\rightarrow\R$ is a \textbf{cut-off function}. In this paper, unless otherwise mentioned, we will assume $\chi(y)=\mathbbm{1}_{(0,1)}(|y|)$. For each $x$, $(b(x),\Sigma(x),\nu(x,dy))$ is the \textbf{(L\'{e}vy) characteristic triplet}, where $b(x)\in\R^d$, $\Sigma(x) \in \R^{d\times d}$ a symmetric positive definite matrix, and $\nu(x,dy)$, the \textbf{\Levy measure}, is a $\sigma$-finite measure on $\R^d\setminus\{0\}$ satisfying $\int_{\R^d\setminus\{0\}} (1\wedge |y|^2) \nu(x,dy)<\infty$. We call the function $p(x,\xi)$ the \textbf{symbol} of the process. We also write $X\sim (b(x),\Sigma(x), \nu(x,dy))$ to signify that $X$ is a Feller process with that characteristic triplet. 
\par
 When the symbol, and the corresponding  triplet, are constant in $x$, i.e. $p(x,\xi) = p(\xi)$ and triplet $(b(x),\Sigma(x),\nu(x,dy)) = (b,\Sigma,\nu)$ then  process $X$ is a \textbf{\Levy process}, i.e. a stochastically continuous Markov process with stationary and independent increments. The symbol $p(\xi)$ is also the \Levy symbol of the process, with characteristic function $\phi_{X_t} (\xi) = e^{tp(\xi)}$. In the \Levy case, $b$ is the non-random linear drift, $\Sigma$ is covariance of the Brownian motion, and $\nu$ is a measure representing the jumps of the process.
\par
Continuous negative definite functions $p(x,\xi)$ which are associated with a Feller process have  a form of local boundedness in the first argument. In other words, we say the symbol $p(x,\xi)$ is locally bounded if for all $K\subset\R^d$ compact, there exists $c_K>0$ such that 
\begin{equation}
\label{locbdd}
\sup_{x\in K} |p(x,\xi)| \leq c_K (1 + |\xi|^2).
\end{equation}
We say the \textbf{symbol is bounded} if (\ref{locbdd}) holds for $K=\R^d$. The local boundedness (or boundedness) of the symbol corresponds to the local boundedness (boundedness) of the characteristics $(b(x), \Sigma(x),\nu(x,dy))$ (See \cite[Lem.2.1]{Schilling1998}).

\subsubsection{Integro-differential operator}

For a general rich Feller process, the triplet $(b(x),\Sigma(x), \nu(x,dy))$  characterizes the behavior of the process, with $b(x)$ representing non-random continuous behavior, $\Sigma(x)$ representing the diffusion-like continuous behavior, and $\nu(x,dy)$ representing the jump behavior. To analyze the process one of the crucial tools we will use is the extended generator. For the case of rich Feller processes, when we substitute (\ref{symbol})  into  the right-hand side of (\ref{pseudo}), by elementary Fourier analysis, we get an \textbf{integro-differential operator} $I(p)$,
\begin{equation}
\label{integro}
I(p) f(x) = b(x) \ccdot \nabla f(x) + \frac{1}{2} \nabla\ccdot \Sigma(x)\nabla f(x) +\int_{y\neq0} \left(f(x+y)-f(x) - y\ccdot \nabla f(x) \chi(y) \right)\nu(x,dy)
\end{equation}
where $\nabla\cdot \Sigma(x)\nabla f(x) = \sum_{j,k=1}^d \Sigma_{jk}(x)\partial_j\partial_k f(x)$.
Clearly, the operator $I(p)$ is defined on $C_b^2(\R^d)$, the space of continuous twice-differentiable bounded functions. When the symbol $p(x,\xi)$ is bounded, $I(p)$ is an extension of $-p(x,D)$: 
\begin{center}
$I(p)|_{C_c^\infty(\R^d)}= -p(x,D) = \mcl{A}|_{C_c^\infty(\R^d)}$
\end{center}
and an extension of generator $\mcl{A}$:
$I(p)|_{\mcl{D}(\mcl{A})}= \mcl{A}$,
as shown by Schilling \cite[Lem.2.3]{Schilling1998}. Our interest in this integro-differential operator $I(p)$ comes with wanting to use the idea of Liggett's characterization of association via the generator.

\begin{theorem}[Liggett (1985) \cite{Liggett1985}, p.80]
\label{liggettthm}
Let $X = (X_t)_{t\geq0}$ be a Feller process on state space $E$ with generator $(\mcl{A},\mcl{D}(\mcl{A}))$ and semigroup $(T_t)_{t\geq0}$. If $X$ is stochastically monotone, 
then 
\begin{equation}
\label{liggett}
\mcl{A} fg \geq g\mcl{A} f + f\mcl{A} g, \hspace{.5cm} \forall f,g\in\mcl{F}_i\cap \mcl{D}(\mcl{A})
\end{equation}
if and only if $X_t \text{ is associated for all } t\geq0 \text{ wrt } \dP^x \text{ for all } x\in E.$
\end{theorem}

Liggett proved this for $E$ compact and $\mcl{A}$ bounded. This was extended by Szekli and R\"{u}schendorf to more general Polish spaces $E$ and $\mcl{A}$ unbounded \cite[Ch.3.7]{Szekli1995}, \cite[Cor.3.1]{Ruschendorf2008}. For the Feller processes we consider in the above setting, particularly those of the jump-variety, the domain $\mcl{D}(\mcl{A})$ is often defined to be a dense subspace of $C_0(\R^d)$, and thus, $\mcl{D}(\mcl{A})\cap \mcl{F}_i = \{f\equiv 0\}$. Hence, in that case,  inequality (\ref{liggett}) would always hold. Thus, we would like to extend Theorem \ref{liggettthm} to the extended generator $I(p)$. 

\subsubsection{Small-time asymptotics}

The (extended) generator gives us a connection between the notion of association and the \Levy characteristics $(b(x), \Sigma(x), \nu(x,dy))$ due to the representation of integro-differential operator. Thus, to characterize association for Feller processes using the \Levy characteristics, an extension of Theorem \ref{liggettthm} becomes quite useful. However, under weaker conditions of the symbol $p(x,\xi)$, such as local boundedness, it is useful to surpass the generator (as we will show in Section 3) and show a more direct connection between the \Levy characteristics and the notion of association. We will establish such a connection by looking at small-time asymptotics of a Feller process. Additionally, this notion will allow us to characterize weaker forms of positive dependence under the \Levy characteristics. 

\par
The classical results of small-time asymptotics have been primarily established for \Levy processes. For a given \Levy process $L = (L_t)_{t\geq0}$ it is known that for all $f\in C_c(\R^d\setminus\{0\})$, 
\begin{equation}
\label{levysmalltime}
\lim_{t\searrow0} \frac{1}{t} \E^0 f(L_t) = \int_{\R^d\setminus\{0\}} f(y) \nu(dy).
\end{equation}
(See \cite[p.2]{KuhnLaplace2016} for reference.) Thus, by the Portmanteau theorem, (\ref{levysmalltime}) implies  \begin{center}
$\displaystyle\lim_{t\searrow0}\frac{1}{t} \dP^0 (L_t \in A)=\nu(A)$
\end{center}
for all $A\in\mcl{B}(\R^d\setminus\{0\})$ with $0\notin \overline{A}$ and $\nu(\partial A)=0$. This result naturally extends to a general starting point $x$: For every $x\in\R^d$, $\lim_{t\searrow0}\frac{1}{t} \dP^x (L_t -x\in A)=\nu(A)$
by translation invariance of a \Levy process. Until recently, an analogous statement of the above for Feller processes was not known. However, K\"{u}hn and Schilling (2016) proved in \cite{KuhnLaplace2016} 
such a statement for such processes.

\begin{theorem}[K\"{u}hn, Schilling (2016) \cite{KuhnLaplace2016}, Cor.3.3]
\label{kuhn}
Let $X = (X_t)_{t\geq0}$ be a rich Feller process with symbol $p(x,\xi)$ and characteristics $(b(x), \Sigma(x), \nu(x,dy))$. If $f\in C_0(\R^d)$ and $f|_{B(0,\delta)} = 0$ for some $\delta>0$, then $$\lim_{t\searrow0} \frac{1}{t} \E^x f(X_t -x) = \int_{\R^d\setminus\{0\}} f(y) \nu(x,dy).$$
Additionally, by the Portmanteau theorem, 
\begin{equation*}
\label{fellersmalltime}
\lim_{t\searrow0}\frac{1}{t} \dP^x(X_t - x \in A) = \nu(x,A)
\end{equation*}
for all $A\in\mcl{B}(\R^d\setminus\{0\})$ such that $0\notin \overline{A}$ and $\nu(x,\partial A)=0$.
\end{theorem}

The small-time asymptotics given by Theorem \ref{kuhn} give us a direct connection between the \Levy measure and the Feller process, surpassing the representation of the generator. Also, notice that the result holds for more general, locally bounded symbols. 

\par

Our interest focuses on jump-Feller processes, i.e. $X\sim (b(x), 0, \nu(x,dy))$, since the association of diffusion processes $X\sim (b(x), \Sigma(x), 0)$ has been done by Mu-Fa Chen \cite{Chen1993}. In the following section, we will prove a sufficient and necessary condition for the jump-Feller process to be associated, WA, PSA, PSD, POD, PUOD, and PLOD in space, where the condition is 
\begin{equation}
\label{resnickx}
\nu(x, (\R_+^d \cup \R_-^d)^c) = 0, \hspace{1cm} \forall x\in\R^d.
\end{equation}

\begin{remark}
\label{rem:jmwang}
{\rm We do note that Jie Ming Wang \cite[Thm.1.4]{Wang2009} proved spatial association is equivalent to \eqref{resnickx} under certain continuity and integrability conditions (unknown to the author at the time). These assumptions include 
\begin{itemize}[itemsep=0em]
\item $b_i,\Sigma_{ij}\in C(\R^d)$, for all $i,j$.
\item $\int h_i(z) (\nu(\cdot,dz) - \nu(\cdot, d(-z)))\in C(\R^d)$, where $h:\R^d\rightarrow\R^d$ is defined by $h_i(z) = \text{sgn}(z_i)(1\wedge |z_i|)$.
\item $\int_A |h(z)|^2 \nu(\cdot, dz)\in C(\R^d)$ for all $A\in\mcl{B}(\R^d)$.
\item $\int g(z) \nu(\cdot, dz) \in C(\R^d)$ for all $g\in C_b(\R^d)$ that is 0 near the origin.
\end{itemize}
\noindent
We relax these conditions, and furthermore our work includes characterizations of the other dependence structures mentioned in Definitions \ref{def:assoc}-\ref{def:pod}.
}
\end{remark}

\section{Main results}
\label{sec:mainresults}

Consider a rich Feller process $X = (X_t)_{t\geq0}$ on the space $(\Omega, \mcl{G}, (\mcl{G}_t)_{t\geq0}, \dP^x)_{x\in\R^d}$ with \Levy \\characteristics $(b(x), 0, \nu(x,dy))$. If we assume that $X$ is stochastically monotone, then condition (\ref{resnickx}) is a necessary and sufficient condition for the association, WA, PSA, PSD, POD, PUOD, and PLOD in space of process $X$. These equivalences can be illustrated in the implication map in Figure \ref{fig:impmapres2}. The dashed arrows are the implications we will prove.


\begin{figure}[h]
  \centering
  \includegraphics[trim = {0cm 0cm 0cm 0cm}, scale=.65]{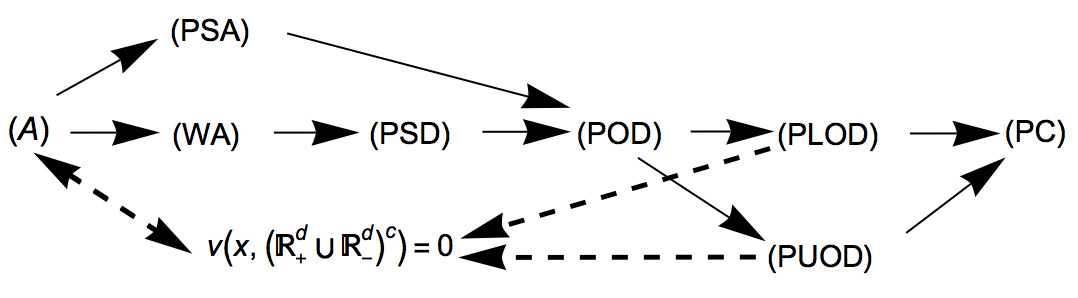}\\
  \caption{Equivalence of dependencies under condition \eqref{resnickx} for Feller processes}\label{fig:impmapres2}
\end{figure}

\par

To show these equivalences, we first give a proof that, under stochastic monotonicity, condition (\ref{resnickx}) is equivalent to association in space. We show this in Section 3.1. Then in Section 3.2, we show that PUOD in space (and, similarly, PLOD) implies  condition (\ref{resnickx}).

\subsection{Association is equivalent to  condition (\ref{resnickx})} 

\begin{theorem}
\label{resnickxthm}
Let $X=(X_t)_{t\geq0}$ be a rich Feller processes with stochastically monotone transition semigroup $(T_t)_{t\geq0}$, a generator $(\mcl{A},\mcl{D}(\mcl{A}))$, bounded symbol $p(x,\xi)$, and $(b(x), 0, \nu(x,dy))$. Then $X_t$ is associated for all $t\geq0$ if and only if condition (\ref{resnickx}): $\nu(x, (\R_+^d \cup \R_-^d)^c)=0$ is satisfied.
\end{theorem}

We prove this by first showing that association of $X_t$'s is equivalent to a Liggett-type inequality for the extended generator, the statement of which is in the following theorem.

\begin{theorem}
\label{extliggett}
Let $X=(X_t)_{t\geq0}$ be a rich Feller processes with stochastically monotone transition semigroup $(T_t)_{t\geq0}$, a generator $(\mcl{A},\mcl{D}(\mcl{A}))$, bounded symbol $p(x,\xi)$, and an (extended) integro-differential operator $I(p)$. Assume $x\mapsto p(x,0)$ is continuous. Then 
\begin{equation}
\label{eq:extliggett}
I(p) fg \geq f I(p) g + g I(p) f, \hspace{.5cm} \forall f,g\in C_b^2(\R^d)\cap \mcl{F}_i
\end{equation}

\noindent if and only if 

\begin{equation}
\label{semiass}
\forall t\geq0, \hspace{.5cm} T_t fg \geq T_t f \cdot  T_t g,  \hspace{.5cm} \forall f,g\in C_b(\R^d)\cap\mcl{F}_i.
\end{equation}
\end{theorem}

Inequality (\ref{semiass}) in Theorem \ref{extliggett} is another way to formulate that $X_t$ is associated for all $t\geq0$. Since inequality (\ref{semiass}) means, for all $x\in\R^d$, $\E^x f(X_t) g(X_t) \geq \E^x f(X_t) \E^x g(X_t)$
which means $X_t$ is associated with respect to $\dP^x$. Inequality \eqref{eq:extliggett} intuitively means that the process moves either up or down, which in multidimensional Euclidean space, means that if the process is currently at point $x$, then it can only move to another point $y$ if $y\geq x$ or $y\leq x$ componentwise. 

\par

Notice that in Theorem \ref{extliggett}, we are using the extended generator $I(p)$. In previous statements of Liggett's characterization, the generator $\mcl{A}$ is used. But we need to use $I(p)$ for the reasons given in the comments after Theorem \ref{liggettthm}. Hence, it is necessary to show the Liggett-type inequality as a characterization of association for rich Feller processes. Such an extension has not been seen by the author of this paper. We first need the following lemmas to prove Theorem \ref{liggettthm}. We will often assume Setting \ref{set1} throughout this section.

\par

\begin{setting} 
\label{set1}
{\rm Let $X=(X_t)_{t\geq0}$ be a rich Feller process with semigroup $(T_t)_{t\geq0}$, generator $(\mcl{A},\mcl{D}(\mcl{A}))$, symbol $p(x,\xi)$, (extended) integro-differential operator $I(p)$, and characteristics $(a(x), b(x), \Sigma(x), \nu(x,dy))$, where $b,\Sigma,\nu$ are the same before, except we have an additional characteristic $a:\R^d \rightarrow \R_+$ which represents the ``killing rate." }
\end{setting}

\begin{remark}
{\rm With the additional characteristic $a(x)$, function $-p(x,\xi)$ would look like  $a(x)+\text{ RHS of equation } (\ref{symbol})$. Also, $I(p)f(x)$ would look like $- a(x) f(x)+ \text{RHS of equation }(\ref{integro})$. Unless stated otherwise, we will assume that $a(x)\equiv0$. For more on the case when $a(x)\not\equiv0$, see the paper by Schnurr \cite{Schnurr2017}, which discusses such processes satisfying $a(x)\not\equiv0$ and their connection to the symbol.} 
\end{remark}




\par

\begin{lemma}
\label{integrogenerates}
Assume Setting \ref{set1} and that $p(x,\xi)$ is bounded. 
Then $I(p)$ generates the semigroup $(T_t)_{t\geq0}$ locally uniformly, i.e. 
\begin{equation}
\label{generates}
I(p) f = \lim_{t\searrow0}\frac{1}{t}(T_t f - f), \hspace{.5cm} f\in C_b^2(\R^d),
\end{equation}
where the convergence is locally uniform.
\end{lemma}

For a detailed proof,  see the Appendix.


\begin{lemma}
\label{integroderivative}
Assume Setting \ref{set1} and the symbol $p(x,\xi)$ is bounded 
For all $f\in C_b^2(\R^d)$,  $$\frac{d}{dt} T_t f = I(p)T_t f = T_t I(p) f$$
where the derivative is defined based on locally uniform convergence. 
\end{lemma}



For a detailed proof, see the Appendix. Finally, we can extend Liggett's solution to a Cauchy problem \cite[Thm.2.15, p.19]{Liggett1985} to integro-differential operators that generate a semigroup locally uniformly.


\begin{lemma}[Cauchy problem]
\label{extcauchy}
Let $(\mcl{A},\mcl{D}(\mcl{A}))$ be a (rich) Feller generator of a semigroup $(T_t)_{t\geq0}$ with bounded \Levy characteristics and symbol $p(x,\xi)$ 
Let $I(p)$ be the extended generator on $C_b^2(\R^d)$. Suppose $F,G:[0,\infty)\rightarrow C_b(\R^d)$ such that 
\mypar

\noindent (a) $F(t)\in\mcl{D}(I(p))$ for all $t\geq0$
\\
(b) $G(t)$ is continuous on $[0,\infty)$ (locally uniformly)
\\ 
(c) $F'(t) = I(p) F(t) + G(t)$ for all $t\geq0$.
\mypar

\noindent Then $\displaystyle F(t) = T_t F(0) + \int_0^t T_{t-s} G(s) ds.$
\end{lemma}

For a detailed proof, see the Appendix. We are now ready to prove the main theorems of this section.
\\ \\
 \textbf{Proof of Theorem \ref{extliggett}}
 
\begin{proof}
$(\Leftarrow)$ Assume $T_t fg \geq T_t f \hspace{.1cm}T_t g$ for all $f,g\in C_b^2(\R^d)\cap \mcl{F}_i$. This implies
\begin{align*}
T_t fg -fg &\geq T_tf \hspace{.1cm} T_tg -fg  =T_tf \hspace{.1cm}T_tg -fg + g \hspace{.1cm}T_t f - g \hspace{.1cm}T_t f= T_t f[T_t g- g] + g[T_t f-f].
\end{align*}
Hence, for all $t>0$,
$\displaystyle\frac{1}{t} (T_t fg - fg) \geq T_t f  \hspace{.1cm}\frac{T_t g- g}{t} + g\hspace{.1cm} \frac{T_t f-f}{t}$.
Therefore, 
\begin{align*}
I(p) fg   = \lim_{t\searrow0} \frac{1}{t} (T_t fg - fg) & \geq \lim_{t\searrow0} \left(T_t f  \hspace{.1cm}\frac{T_t g- g}{t} + g\hspace{.1cm} \frac{T_t f-f}{t}\right)
\\ & = \left(\lim_{t\searrow0} T_t f \right) \hspace{.1cm} \left(\lim_{t\searrow0} \frac{T_t g- g}{t}\right) + g\hspace{.1cm} \left(\lim_{t\searrow0}  \frac{T_t f-f}{t}\right)
\\ & = f I(p) g + g I(p) f,
\end{align*}
where the convergence is locally uniform.
\\
\\
$(\Rightarrow)$ Assume $I(p) fg \geq f I(p) g + gI(p) f$ for all $f,g\in C_b^2(\R^d)\cap \mcl{F}_i$. By monotonicity, $T_t f, T_t g \in C_b^2(\R^d)\cap \mcl{F}_i$, which implies 
\begin{equation}
\label{eq:liggettwithsemi}
I(p) (T_t f) (T_t g) \geq T_tf [I(p)T_t g] + T_t g [I(p) T_tf].
\end{equation}
Define $F(t) := T_t fg -  T_t f \hspace{.1cm}T_t g$. Then by Lemma \ref{integroderivative}, we have 
\begin{align*}
F'(t) 
 = I(p) T_t fg - (T_t f [I(p) T_t g] + T_t g [I(p) T_t f]) &\geq I(p) T_t fg - (I(p) T_t f  \hspace{.1cm} T_tg)
\\ & = I(p) (T_t fg - T_t f\hspace{.1cm} T_tg)
\\ & = I(p) F(t)
\end{align*}
where the inequality comes from \eqref{eq:liggettwithsemi}.
Define $G(t) := F'(t) - I(p) F(t)\geq0$. Then by Lemma \ref{extcauchy}, the solution to the Cauchy problem $F'(t) = G(t) + I(p) F(t)$ is given by 
$$F(t) = T_t F(0) + \int_0^t T_{t-s} G(s) ds =  \int_0^t T_{t-s} G(s) ds$$
since $F(0)=0$. Since $G(s)\geq0$ for all $s$, and $T_{t-s}$ is a positivity-preserving linear operator, $F(t)\geq0$ for all $t\geq0$. Thus, $T_t fg \geq T_t f \cdot T_t g$ for all $f,g\in C_b^2(\R^d)\cap\mcl{F}_i$. This inequality also holds for all $f,g\in C_b(\R^d)\cap \mcl{F}_i$, since we can approximate non-decreasing, continuous, bounded functions by non-decreasing smooth, bounded functions, and then use a dominated convergence argument.
\end{proof}

\begin{remark}
{\rm For the necessary condition, we did not need stochastic monotonicity.}
\end{remark}

\noindent \textbf{Proof of Theorem \ref{resnickxthm}}

\begin{proof}
$(\Leftarrow)$. Fix $x\in\R^d$. Assume $\nu(x,(\R_+^d\cup\R_-^d)^c)=0$. Then, for all $f,g\in C_b^2(\R^d)\cap\mcl{F}_i$,
\begin{align*}
& I(p) fg(x) - g(x) I(p) f(x) - f(x)I(p) g(x) 
\\ & = b(x)\cdot \nabla fg(x)+ \int_{y\neq0} \left(f(x+y) g(x+y)- f(x)g(x) - y\cdot \nabla fg(x)\mathbbm{1}_{(0,1)}(|y|)\right)\nu(x,dy)
\\ & \hspace{.2cm} -b(x)\cdot g(x)\nabla f(x) - \int_{y\neq0} \left(f(x+y) g(x)- f(x)g(x) - y\cdot g(x) \nabla f(x)\mathbbm{1}_{(0,1)}(|y|) \right)\nu(x,dy)
\\ & \hspace{.2cm} -b(x)\cdot f(x) \nabla g(x)-  \int_{y\neq0} \left(f(x) g(x+y)- f(x)g(x) -y\cdot f(x)\nabla g(x) \mathbbm{1}_{(0,1)}(|y|) \right)\nu(x,dy)
\\ & = \int_{y\neq0} \left(f(x+y)g(x+y) - f(x+y) g(x) - f(x)g(x+y) +f(x)g(x) \right)\nu(x,dy)
\\ & = \int_{y\neq0} \left(f(x+y)- f(x))(g(x+y) - g(x)\right)\nu(x,dy)
\\ & = \int_{\R_+^d} \left(f(x+y)- f(x))(g(x+y) - g(x)\right)\nu(x,dy) 
\\ & \hspace{.2cm} + \int_{\R_-^d} \left(f(x+y)- f(x))(g(x+y) - g(x)\right)\nu(x,dy)
\\ & \geq0,
\end{align*}
where the drift terms and the cut-off term in the integrand vanish because $\nabla fg(x) = f(x) \nabla g(x) + g(x) \nabla f(x)$. Additionally, we get positivity at the end there because $\forall y\in\R_+^d$, $f(x+y) - f(x)\geq0$, and $g(x+y)-g(x)\geq0$, so $(f(x+y) - f(x))(g(x+y) - g(x))\geq0$ on $\R_+^d$. A similar result holds on $\R_-^d$. By Theorem \ref{extliggett}, this implies $T_t fg(x) \geq T_t f(x) T_tg(x)$, where $f,g\in C_b^2(\R^d)\cap \mcl{F}_i$.
Now to obtain association of $X_t$, this inequality needs to hold for all $f,g\in C_b(\R^d)\cap \mcl{F}_i$  But we can use an approximation of a function $f\in C_b(\R^d)\cap \mcl{F}_i$ by $f_n\in C_b^\infty(\R^d)\cap \mcl{F}_i$ which gives us the desired result. 

\par
$(\Rightarrow)$. Assume $X_t$ is associated for all $t\geq0$. This means $T_t fg(x) \geq T_tf(x) T_tg(x)$ for all $x\in\R^d$, for all $f,g\in C_b(\R^d)\cap \mcl{F}_i$. So this inequality of course holds for $f,g\in C_b^2(\R^d)\cap \mcl{F}_i$, which yields $I(p)fg \geq g I(p) f+ f I(p) g$ for such $f,g$ by Theorem \ref{extliggett}. This implies, by a similar calculation in the $(\Leftarrow)$ direction, that $$\int_{y\neq0} (f(x+y)-f(x))(g(x+y)-g(x))\nu(x,dy)\geq0.$$ 

For simplicity, assume $d=2$, but know that we can easily generalize this result to higher dimensions using correction functions.  Fix $x=(x_1,x_2)\in\R^2$. Assume for contradiction that Resnick's condition is not satisfied. WLOG, let's say $\nu(x,(0,\infty)\times (-\infty,0))>0$. By continuity of measure, $\exists a>0$ such that $\nu(x,(a,\infty)\times (-\infty,a))>0$. Let $\epsilon\in(0,1)$, and define $f,g\in C_b^\infty(\R^2)\cap \mcl{F}_i$ by 
\begin{equation*}
 f(y_1,y_2)= \begin{cases} 
      0 & \textrm{ if \hspace{.1cm} $y_1\leq x_1+\epsilon a$} \\
      1 & \textrm{ if \hspace{.1cm} $y_1\geq x_1+ a,$} \\
   \end{cases} \quad \quad \quad
 g(y_1,y_2)= \begin{cases} 
      0 & \textrm{ if \hspace{.1cm} $y_2\geq x_2-\epsilon a$} \\
      -1 & \textrm{ if \hspace{.1cm} $y_2\leq x_2-a.$} \\
   \end{cases} 
\end{equation*}
\noindent This implies $f(x)=g(x)=0$. Hence,
\begin{align*}
0 & \leq \int_{y\neq0} (f(x+y)- f(x))(g(x+y)-g(x)) \nu(x,dy)
\\ & = \int_{y\neq0} f(x+y)g(x+y) \nu(x,dy)
\\& = \int_{(a,\infty)\times (-\infty,-a)} f(x+y)g(x+y) \nu(x,dy)  + \int_{ (a,\infty)\times [-a,-\epsilon a]} f(x+y)g(x+y) \nu(x,dy)
\\ & \hspace{.2cm} + \int_{ [\epsilon a,a] \times (-\infty,-a)} f(x+y)g(x+y) \nu(x,dy) + \int_{ [\epsilon a,a] \times [-a,-\epsilon a]} f(x+y)g(x+y) \nu(x,dy)
\\ & = -\nu(x, (a,\infty)\times (-\infty,-a)) - \int_{(a,\infty)\times [-a,-\epsilon a]} g(x+y) \nu(x,dy)
\\ & \hspace{.2cm} +\int_{ [\epsilon a,a] \times (-\infty,-a)} f(x+y)\nu(x,dy)  + \int_{[\epsilon a,a] \times [-a,-\epsilon a]} f(x+y)g(x+y) \nu(x,dy)
\\ &\leq -\nu(x, (a,\infty)\times (-\infty,-a)),
\end{align*}   
implying $\nu(x,(a,\infty)\times (-\infty,-a)) \leq0$. Hence, $\nu(x,(a,\infty)\times (-\infty,-a)) =0$, a contradiction. 
\end{proof}

\subsection{PUOD implies  condition (\ref{resnickx})}

\begin{lemma}
\label{PODsub}
If $Y= (Y_1,...,Y_d)$ is PUOD, then $(Y_{k_1},...,Y_{k_n})$ is PUOD for all multi-indices $\{k_j\}_{j=1}^n\subset \{1,...,d\}$. 
\end{lemma}

\begin{proof}
If $Y$ PUOD, then we know $\E \left(\prod_{i=1}^d f_i (Y_i) \right) \geq \prod_{i=1}^d \E f_i (Y_i)$ where $f_i:\R\rightarrow\R_+$ non-decreasing. So for all $i\in\{1,...,d\} \setminus \{k_j\}_{j=1}^n$, set $f_i = \mathbbm{1}_{\R}$. Then the above inequality becomes 
\begin{center}
$\E \left(\prod_{j=1}^n f_j (Y_{k_j}) \right) \geq \prod_{j=1}^n \E f_j (Y_{k_j}).$
\end{center}
Thus, we have that $(Y_{k_1},...,Y_{k_n})$ is PUOD.
\end{proof}




\begin{theorem}
\label{PODnec}
Let $X = (X_t)_{t\geq0}$ be a rich Feller process with symbol $p(x,\xi)$ and triplet $(b(x), 0, \nu(x,dy))$. Then, $X_t$ is  PUOD for each $t\geq0$ implies condition (\ref{resnickx}): 
\begin{center}
$\nu(x,(\R_+^d\cup \R_-^d)^c)=0.$
\end{center}
\end{theorem}

\begin{proof}
Assume $X_t$ is PUOD (wrt $\dP^x$) for each $t\geq0$.  Fix $x=(x_1,...,x_d)\in\R^d$. Since $X_t$ is PUOD, then $X_t-x$ is PUOD for all $t\geq0$. Assume for contradiction that $\nu$ not concentrated on $\R_+^d\cup \R_-^d$. WLOG, say $\nu(x,(0,\infty)^{d-1} \times (-\infty,0))>0$. By continuity of measure 
there exists $a>0$ such that 
\begin{center}
$\nu(x,(a,\infty)^{d-1}\times (-\infty,-a))>0$
\end{center}
 and 
 \begin{center}
 $\nu(x, \partial[(a,\infty)^{d-1}\times (-\infty,-a)])=  \nu(x,\partial[(a,\infty)\times \R^{d-1}])=0.$
 \end{center}
Then by Theorem \ref{kuhn}, 
\begin{center}$\lim_{t\rightarrow0} \frac{1}{t} \dP^x( X_t - x\in (a,\infty)^{d-1} \times (-\infty,-a)) = \nu(x, (a,\infty)^{d-1} \times (-\infty,-a)).$\end{center}
Hence,
\begin{align*}
0&<\nu(x, (a,\infty)^{d-1}\times (-\infty,-a))
\\ & = \lim_{t\rightarrow0}\frac{1}{t} \dP^x(X_t - x \in (a,\infty)^{d-1}\times (-\infty,-a))
\\ & = \lim_{t\rightarrow0}\frac{1}{t} \dP^x(X_t^{(1)} - x_1>a ,...,X_t^{(d-1)} - x_{d-1}>a, X_t^{(d)} - x_d <-a)
\\ & \leq \lim_{t\rightarrow0}\frac{1}{t} \dP^x(X_t^{(1)} - x_1>a ,...,X_t^{(d-1)} - x_{d-1}>a, X_t^{(d)} - x_d \leq-a)
\\ & = \lim_{t\rightarrow0}\frac{1}{t} \dP^x( \{X_t^{(1)} - x_1>a\}\setminus [\{X_t^{(1)} - x_1>a\} \cap  \{X_t^{(2)} - x_2 >a,..., X_t^{(d)} - x_d \leq-a\}^c])
\\ & = \lim_{t\rightarrow0}\frac{1}{t} \left[ \dP^x( X_t^{(1)} - x_1>a) \right.
\\ & \hspace{1cm} \left. - \dP^x (\{X_t^{(1)} - x_1>a\} \cap  \{X_t^{(2)} - x_2 >a,..., X_t^{(d)} - x_d \leq-a\}^c ) \right]
\\ & = \lim_{t\rightarrow0}\frac{1}{t} \left[\dP^x( X_t^{(1)} - x_1>a) \right.
  \\ & \hspace{1cm}  - \left. \dP^x(\{X_t^{(1)} - x_1>a\} \cap  [ \{X_t^{(2)} - x_2 \leq a\}\cup...\cup  \{X_t^{(d)} - x_d >-a\}]]) \right]
\\ & =  \lim_{t\rightarrow0}\frac{1}{t} \left[ \dP^x(X_t^{(1)} - x_1>a) \right.
\\ &  \hspace{1cm}  - \left. \dP^x( \{X_t^{(1)} - x_1>a, X_t^{(2)} - x_2 \leq a\}\cup...\cup  \{X_t^{(1)} - x_1>a, X_t^{(d)} - x_d >-a\}]) \right]
\\ & \leq \lim_{t\rightarrow0}\frac{1}{t} \left[ \dP^x(X_t^{(1)} - x_1>a)  -  \dP^x(   X_t^{(1)} - x_1>a, X_t^{(d)} - x_d >-a]) \right]
\\ & \leq \lim_{t\rightarrow0}\frac{1}{t} \left[ \dP^x(X_t^{(1)} - x_1>a)  -  \dP^x(   X_t^{(1)} - x_1>a)\dP^x(X_t^{(d)} - x_d >-a]) \right]
\\ & = \lim_{t\rightarrow0}\frac{1}{t} \left[ \dP^x(X_t^{(1)} - x_1>a)(1  - \dP^x(X_t^{(d)} - x_d >-a]))\right]
\\ & = \lim_{t\rightarrow0}\frac{1}{t} \dP^x(X_t^{(1)} - x_1>a)\dP^x(X_t^{(d)} - x_d \leq -a])
\\ & = \left[\lim_{t\rightarrow0}\frac{1}{t} \dP^x(X_t^{(1)} - x_1>a)\right] \left[\lim_{t\rightarrow0}\dP^x(X_t^{(d)} - x_d \leq -a])\right]
\\ & = \nu(x, (a,\infty)\times \R^{d-1}) \dP^x( X_0^{(d)} - x_d \leq -a)
\\ & = 0.
\end{align*}
We obtain lines 4, 9 by set containment, line 5 by the fact: $A\cap B = A\setminus (A\cap B^c)$, line 10 by Lemma \ref{PODsub}, and line 14 by Theorem \ref{kuhn}. This contradiction gives us the desired result.
\end{proof}

\begin{remark}
{\rm 
\begin{enumerate}[noitemsep, label=(\roman*)]
\item We could have also showed PLOD implies condition \eqref{resnickx} using similar techniques to those above.
\item Symbol $p(x,\xi)$ in the above theorem need not be bounded, only locally bounded. 
\end{enumerate}
}
\end{remark}

\begin{corollary}
For stochastically monotone jump-Feller processes, i.e. $X\sim (b(x), 0, \nu(x,dy))$ with bounded symbols $p(x,\xi)$, then condition (\ref{resnickx}), $\nu(x,(\R_+^d\cup\R_-^d)^c)=0$, is equivalent to $X$ being associated, WA, PSA, PSD, POD, PUOD, and PLOD in space.

\end{corollary}
\begin{proof}
True by Theorems \ref{resnickxthm} and \ref{PODnec}.
\end{proof}

\subsection{Association in time}

Our results can also be applied to study the temporal association of Feller processes. 
 We first examine the case of \Levy processes, a sub-class of Feller processes with constant characteristic triplet $(b,Q,\nu)$. For \Levy processes, spatial association is equivalent to temporal association. 

\begin{theorem}
\label{levytime}
Let $X=(X_t)_{t\geq0}$ be a stochastic process in $\R^d$ with independent and stationary increments, i.e. $X_t - X_s \independent X_s - X_r$, for all $0\leq r<s<t$, and $X_t-X_s \stackrel{d}=X_{t-s}$ for all $0\leq s<t$. Then $X$ is associated in time if and only if $X$ is associated in space.
\end{theorem}

\begin{proof}
The forward direction is trivial by definition. We only need to prove the backward direction. Assume $X_t$ is associated in $\R^d$ for every $t\geq0$. Choose $0\leq t_1<...<t_n$. Then 
\begin{align*}
(X_{t_1},...,X_{t_n}) & = (X_{t_1}, X_{t_1} + (X_{t_2} - X_{t_1}), ..., X_{t_1} + (X_{t_2} - X_{t_1}) +...+(X_{t_n} - X_{t_{n-1}}))
\\ & = (X_{t_1},...,X_{t_1}) + (0, X_{t_2} - X_{t_1},...,X_{t_2}-X_{t_1}) + ... + (0,...,0,X_{t_n} - X_{t_{n-1}})
\end{align*}
Now observe that by stationary increments, $X_{t_{k+1}} - X_{t_{k}} \stackrel{d}= X_{t_{k+1} - t_{k}}$ and   $X_{t_{k+1} - t_{k}}$ is associated, which makes $X_{t_{k+1}} - X_{t_{k}}$ associated (association is preserved under equality in distribution), for all $k\in\{1,...,n-1\}$. Now observe that if $\hat{X}$ is associated in $\R^d$, then each block $(0,...0,\hat{X},...,\hat{X})$ is associated in $\R^{dn}$, where there are a $k$ number of $0$ vectors  and $(n-k)$ $\hat{X}$ vectors. Therefore, each block $(0,...,0, X_{t_{k+1}} - X_{t_k},...,X_{t_{k+1}} - X_{t_k})$ is associated, for each $k\in\{1,...,n-1\}$. By independent increments, each block is independent. Therefore, since the sum of independent random vectors, each of which is associated, is associated, then $(X_{t_1},...,X_{t_n})$ is associated. 
\end{proof}

\begin{corollary}
Any \Levy process $X$ that is associated in space is also associated in time. Additionally, if $X$ has triplet $(b,0,\nu)$, then $X$ is associated in time if and only if \\ $\nu( (\R_+^d\cup\R_-^d)^c)=0$.
\end{corollary}

\begin{proof}
Any \Levy process has independent and stationary increments, thus the result holds by Theorem \ref{levytime}.
\end{proof}

We would also like to consider conditions for temporal association of general Feller processes. Early work on this has been done by Harris \cite[Cor.1.2]{Harris1977} and Liggett \cite[p.82]{Liggett1985} for Feller processes with a countable state space. This can be extended to more general state spaces, as given in the following theorem. 

\begin{theorem}
\label{thm:liggetttime}
Let $X= (X_t)_{t\geq0}$ be a time-homogeneous, stochastically monotone Feller process on $\R^d$. If $X$ is spatially associated, and $X_0\sim \mu$, where $\mu$ satisfies 
\begin{equation*}
\label{assocmeasure}
\int fg \hspace{.1cm} d\mu \geq \int f \hspace{.1cm} d\mu \int g  \hspace{.1cm} d\mu, \hspace{.5cm} f,g\in B_b(\R^d) \cap \mcl{F}_i,
\end{equation*}
then $X$ is temporally associated.
\end{theorem}

The proof is similar to Liggett's proof found in  \cite[p.82]{Liggett1985}. For details on the proof, we refer the reader to the author's dissertation \cite[p.59]{TuThesis}. Theorem \ref{thm:liggetttime} yields the following corollary about jump-Feller processes.

\begin{corollary}
\label{cor:assoctimeFeller}
 Let $X=(X_t)_{t\geq0}$ be a stochastically monotone Feller process with characteristics $(b(x),0,\nu(x,dy))$. Assume $X_0\sim \mu\in\mcl{M}_a$. Then $\nu(x,(\R_+^d \cup \R_-^d)^c)=0$ if and only if $X$ is associated in time.
\end{corollary}

\begin{proof}
The proof follows from Theorems \ref{resnickxthm} and  \ref{thm:liggetttime}.
\end{proof}

\section{Examples}
\label{sec:examples}
 
We give a collection of interesting Feller processes that satisfy stochastic monotonicity.

\subsection{\Levy processes}
Any \Levy process satisfies stochastic monotonicity. Let $(T_t)_{t\geq0}$ be a semigroup of a \Levy process. Then, for $f\in\mcl{F}_i$, we have $$T_t f(x) = \E^x f(X_t) = \E^0 f(X_t+x).$$
Thus monotonicity of function $f$ and of the expectation $\E^0$ gives us that $T_t f\in\mcl{F}_i$.
\par

Let  $X = (X_t)_{t\geq0}$  be a jump-\Levy process whose \Levy characteristics look like $(b, 0, \nu)$, where there is no state-space dependence. Then $\nu((\R_+^d \cup \R_-^d)^c)=0$ is equivalent to $X_t$ being associated, WA, PSA, PSD,  POD, PUOD, and PLOD since all \Levy processes are stochastically monotone. This was proven in B\"{a}uerle (2008) \cite{Bauerle2008} for association, PSD, and POD, but not for the other dependence structures. Furthermore, the technique in \cite{Bauerle2008} to prove condition (\ref{resnick}) is equivalent to PSD and POD required \Levy copulas. Our method of short-time asymptotics avoids \Levy copulas altogether, and solely uses the \Levy measure. Additionally, condition \eqref{resnick} is equivalent to temporal association of $X$, by Corollary \ref{cor:assoctimeFeller}.




\subsection{Ornstein-Uhlenbeck process}

An Ornstein-Uhlenbeck (OU) process $X=(X_t)_{t\geq0}$ in $\R^d$ is the solution to the general \textit{Langevin equation}: 
\begin{align*}
dX_t &= -\lambda X_t dt + dL_t
\\ X_0 &=x \text{ a.s.}
\end{align*}

\noindent where $\lambda>0$, $L = (L_t)_{t\geq0}\sim(b_L,\Sigma_L,\nu_L)$ is a \Levy process in $\R^d$, and $x\in\R^d$. Then OU-process looks like:
$$X_t = e^{-\lambda t}x + \int_0^t e^{-\lambda(t-s)} dL_t$$
The semigroup $(T_t)_{t\geq0}$ of this process is called the Mehler semigroup and is given by $$T_t f(x) = \int_{\R^d} f( e^{\lambda t} x + y) \mu_t(dy), \hspace{.5cm} L_t \sim \mu_t$$

\begin{claim}
The OU process is stochastically monotone.
\end{claim}

\begin{proof}
Let $f\in B_b(\R^d)$ be an increasing function. Assume $x<y$, and fix some $t\geq0$. Then $e^{-\lambda t} x < e^{-\lambda t}y$. This implies $f(e^{-\lambda t} x + z) \leq f(e^{-\lambda t} y + z)$ for all $z\in \R^d$. Hence, 
\begin{center}
$T_t f(x) = \int_{\R^d} f(e^{-\lambda t} x+ z) \mu_t(dz) \leq \int_{\R^d} f(e^{-\lambda t} y+ z) \mu_t(dz) = T_t f(y).$
\end{center}
Thus, $T_t f$ is an increasing function on $\R^d$. 
\end{proof}

Process $X$ has characteristic triplet: $(b_L-\lambda x, \Sigma_L, \nu_L)$ \cite{Applebaum2007}. Thus, the characterization of positive dependence (association, WA, PSA, PSD,  POD, PUOD, PLOD) is equivalent to $\nu_L((\R_+^d \cup \R_-^d)^c)=0$ when $\Sigma=0$.

\subsection{Feller's pseudo-Poisson process}
Here we construct a stochastically monotone pseudo-Poisson process. Let $S=(S(n))_{n\in\N}$ be a homogeneous Markov process taking values in $\R^d$. Let $(q^{(n)})_{n\in\N}$ define the $n$-step transition probabilities:
$$q^{(n)} (x,B) = \dP(S(n)\in B| S(0)=x)$$
for all $B\in \mcl{B}(\R^d)$. Let $Q$ be the \textit{transition operator} of $S$, defined by $$(Q f)(x) = \int_{\R^d} f(y) q(x,dy)$$ for all $f\in B_b(\R^d)$, $x\in\R^d$. Note that $Q^n f(x) = \int_{\R^d} f(y) q^{(n)} (x,dy)$. Let $N = (N_t)_{t\geq0}$ be a Poisson process with rate $\lambda$ that is independent of $S$. Define $X = (X_t)_{t\geq0}$ by subordination:
$$X_t := S(N_t) \text{ for all } t\geq0.$$ Process $X$, called \textit{Feller's pseudo-Poisson process},  is a Feller process.
The semigroup $(T_t)_{t\geq0}$ and generator $\mcl{A}$ of $X$ are given by:
$$T_t f(x) = e^{t[\lambda(Q-I)]} f(x) = e^{-\lambda t}\sum_{n=0}^\infty \frac{(\lambda t)^n}{n!} Q^n f(x),$$

$$\mcl{A} f(x) = \lambda (Q- I) f(x) = \int_{\R^d} [f(y) - f(x)] \lambda q(x,dy)$$


\begin{claim}
If $S$ is a stochastically monotone Markov process, then $X$ is stochastically monotone. 
\end{claim}

\begin{proof}
We will show that for $f\in\mcl{F}_i$, we have $T_t f\in\mcl{F}_i$. Observe that by $S$ stochastically monotone, we have that $q(x,B)$ is monotone function in $x$ for all $B\in\mcl{B}(\R^d)$ monotone set. Additionally, we have for $f\in B_b(\R^d)\cap \mcl{F}_i$, $Q f(x) = \int_{\R^d} f(y) q(x,dy)$ is a monotone function. We  show, by induction, that for all $n$, $G_n: = e^{-\lambda t} \frac{(\lambda t)^n}{n!} Q^n f$ is a non-decreasing function. 
\\ \\
\underline{Base Case}: $\underline{n=0}$: $G_0(x) = e^{-\lambda t} f(x)$ is non-decreasing. $\underline{n=1}$:  $G_1(x)  = e^{-\lambda t} \lambda t \hspace{.1cm} Q f(x) =e^{-\lambda t} \lambda t \int_{\R^d} f(z) q(x,dz) $ is non-decreasing.\\
\underline{Induction Hypothesis}: Assume $G_n(x) = e^{-\lambda t} \frac{(\lambda t)^n}{n!} Q^n f(x) =e^{-\lambda t} \frac{(\lambda t)^n}{n!} \int_{\R^d} f(z) q^{(n)}(x,dz) $ is a non-decreasing function. \\
\underline{Inductive Step}: 
\begin{align*}
G_{n+1}(x)  = e^{-\lambda t} \frac{(\lambda t)^{n+1}}{(n+1)!} Q^{n+1} f(x) & = e^{-\lambda t} \frac{(\lambda t)^{n+1}}{(n+1)!}  \int_{\R^d} f(z) \hspace{.1cm}q^{(n+1)}(x,dz)
\\ & = e^{-\lambda t} \frac{(\lambda t)^{n+1}}{(n+1)!} \int_{\R^d}\left( \int_{\R^d} f(z) \hspace{.1cm} q^{(n)}(y,dz) \right) q(x,dy)
\\ & =:e^{-\lambda t} \frac{(\lambda t)^{n+1}}{(n+1)!}  \int_{\R^d}H(y) q(x,dy)
\end{align*}
where $H(y) = \int_{\R^d} f(z) q^{(n)}(y,dz)$ is a non-decreasing function in $y$ by Induction Hypothesis, and line 2 is obtained by Chapman-Kolmogorov equations. Thus, by Base Case, the integral $\int_{\R^d} H(y) q(x,dy)$ is non-decreasing in $x$.
Hence we get $G_n$ is a non-decreasing function for all $n$. Hence, $T_tf$ is non-decreasing, giving us our desired result.
\end{proof}


\noindent Now to find the characteristic triplet $(b(x), \Sigma(x), \nu(x,dy))$, we consider the generator: 
\begin{align*}
\mcl{A} f(x) & = \int_{\R^d} (f(z) - f(x)) \lambda q(x,dz)=  \int_{\R^d} (f(x+z) - f(x)) \lambda q(x,dz+x)
\\ & = \int_{\R^d} (f(x+z) - f(x)) \lambda \hat{q}(x, dz), \hspace{.1cm} \text{ where } \hat{q}(x,B) := q(x, B+x)
\\ & = \int_{\R^d} (f(x+z) - f(x) - \nabla f(x) \cdot z \chi(z)) \lambda \hat{q}(x, dz) +   \int_{\R^d} \nabla f(x) \cdot z \chi(z) \lambda \hat{q}(x, dz)
\\ & = \int_{\R^d} (f(x+z) - f(x) - \nabla f(x) \cdot z \chi(z)) \lambda \hat{q}(x, dz) 
+   \nabla f(x) \cdot \left(\int_{\R^d}  z \chi(z) \lambda \hat{q}(x, dz) \right).
\end{align*}

\noindent Thus, the \Levy triplet will be $(b(x), \Sigma(x), \nu(x,dy))$, where 
\begin{center}
$b(x)  = \int_{\R^d}  z \chi(z) \lambda \hat{q}(x, dz), \hspace{1cm} \Sigma(x)  = 0,  \hspace{1cm} \nu(x, A)  = \lambda \hat{q}(x,A) = \lambda q(x, A+x).$
\end{center}




\subsection{Bochner's subordination of a Feller process}

Consider a continuous-time Feller process $Y = (Y(t))_{t\geq0}$ with semigroup $(T_t)_{t\geq0}$ and generator $(\mcl{A},\mcl{D}(\mcl{A}))$. Let $N = N_t$ be a subordinator independent of $Y$ with \Levy characteristics $(b,\lambda)$, i.e. has \Levy symbol $\eta(u) = ibu +\int_0^\infty (e^{iuy} - 1)\lambda(dy)$,
where $\E e^{iuN_t} = e^{t\eta(u)}$. Additionally, we can attain  a Laplace transform of the subordinator, $\E e^{-uN_t} = e^{-t\psi(u)}$, where 
\begin{center}
$\psi(u) := - \eta(iu) = bu +\int_0^\infty (1-e^{-uy}) \lambda(dy).$
\end{center}
Function $\psi$ is called the \textbf{Laplace symbol} or \textbf{Bernstein function} of the subordinator. The following is a theorem of Phillips.

\begin{theorem}[Phillips (1952) \cite{Phillips1952}]
Let $X = (X_t)_{t\geq0}$ be given by the prescription $X_t = Y(N_t)$. Then $X$ is a Feller process with semigroup $(T_t^X)_{t\geq0}$ and generator $(\mcl{A}^X, \mcl{D}(\mcl{A}^X))$, given by $$T_t^X f = \int_0^\infty (T_s f) \hspace{.1cm}\mu_{N_t}(ds), \hspace{1cm} \mcl{A}^X f = b\mcl{A} f + \int_0^\infty (T_s f - f) \lambda (ds).$$
\end{theorem}

\begin{claim}
If $Y$ is a stochastically monotone Feller process with semigroup $(T_t)_{t\geq0}$, i.e. $T_t f \in \mcl{F}_i$ for $f\in\mcl{F}_i$, and $N = (N_t)_{t\geq0}$ is a subordinator, then $X=(X_t)_{t\geq0}$ given by $X_t = Y(N_t)$ is a stochastically monotone Feller process.
\end{claim}

\begin{proof}
We already know that $X$ is Feller with semigroup $(T_t^X)_{t\geq0}$. So choose $f\in\mcl{F}_i\cap C_b(\R^d)$. Then $T_s f \in\mcl{F}_i \cap C_b(\R^d)$ for all $s\geq0$. Choose $x<y$. Then $T_s f(x) \leq T_s f(y)$ for all $s\geq0$. Hence,
\begin{center}
$T_t^X f(x)  = \int_0^\infty (T_s f)(x) \hspace{.1cm} \mu_{N_t}(ds) \leq  \int_0^\infty (T_s f)(y) \hspace{.1cm} \mu_{N_t}(ds) = T_t^X f(y).$
\end{center}
Thus, $T_t^X f\in \mcl{F}_i$. 
\end{proof}




Let $Y$ have symbol $p(x,\xi)$. Then $X = Y(N)$ is a Feller process with symbol $p_X(x,\xi)$ that is given by 
$$p_X(x,\xi) = \psi(p(x,\xi)) + \text{ lower order perturbation}.$$
This ``perturbation" is  ``measured in a suitable scale of anisotropic function spaces" \cite[p.104]{Schilling2013}. 
\par Particularly interesting examples are when $N$ is an $\alpha$-stable subordinator, inverse Gaussian subordinator, and Gamma subordinator, and $Y$ is a  diffusion process $Y\sim (b(x), Q(x), 0)$.

\begin{example}
Let $Y$ be a stochastically monotone diffusion process in $\R^d$. This means $Y$ has \Levy characteristics $(b(x), Q(x),0)$. Mu-Fa Chen and Feng-yu Wang \cite{Chen1993} proved that such a process is stochastically monotone if and only if $q_{ij}(x)$ only depends on $x_i$ and $x_j$, and $b_i(x) \leq b_i(y)$ whenever $x\leq y$ with $x_i=y_i$. The generator of $Y$ is given by:
\begin{center}
$\mcl{A}^Y f(x) = b(x)\cdot \nabla f(x) + \frac{1}{2} \nabla \cdot Q(x) \nabla f(x)$
\end{center}
Let $N$ be $\alpha$-stable subordinator with \Levy characteristics $(0,\lambda)$, where $\lambda(dy) = \frac{\alpha}{\Gamma(1-\alpha)} \frac{1}{y^{1+\alpha}} dy.$ The generator $\mcl{A}^X$ of process $X = Y(N)$ looks like 
\begin{align*}
\mcl{A}^X f(x)  & = \int_0^\infty (T_s f(x) - f(x)) \lambda (ds) =  \int_0^\infty (T_s f(x) - f(x))\frac{\alpha}{\Gamma(1-\alpha)} \frac{1}{s^{1+\alpha}} ds.
\end{align*}

\end{example}

\noindent \textbf{Acknowledgements}: The author would like to thank Dr. Jan Rosinski for his helpful advice and guidance regarding the ideas of this paper.

\section*{Appendix}
This appendix contains proofs of some lemmas from Section \ref{sec:mainresults}. Throughout this appendix, we assume Setting \ref{set1}.

\begin{lemmaA}[Schilling (1998) \cite{SchillingCb}, Thm.4.3]
\label{cbfeller}
Assume  $p(x,\xi)$ is bounded. If $x\mapsto p(x,0)$ is continuous, then $(T_t)_{t\geq0}$ extends to a \textbf{$C_b$-Feller semigroup}, i.e.  satisfies
\begin{enumerate}[noitemsep, label = (\alph*)]
\item $T_t: C_b(\R^d) \rightarrow C_b(\R^d)$,
\item $\lim_{h\searrow0} ||T_{t+h} u - T_t u||_{\infty,K} = 0$ for all $K\subset\R^d$ compact, $u\in C_b(\R^d)$, $t\geq0$, where $||u||_{\infty,K} := \sup_{y\in K}| u(y) |$, i.e. locally uniformly continuous.
\end{enumerate}

\end{lemmaA}
\begin{proof}
For proof, see \cite[p.247]{SchillingCb}.
\end{proof}

\subsection*{Proof of Lemma \ref{integrogenerates}}

\begin{proof}
The process 
\begin{equation*}
M_t^f := f(X_t) - f(X_0) - \int_0^t I(p) f(X_{s-}) ds
\end{equation*}
is, for every $f\in C_b^2(\R^d)$, a martingale with respect to $\dP^x$, for all $x$ (see Schilling \cite[Lemma 3.2, p.579]{Schilling1998}). This implies 
\begin{align*}
0 
 & = \E^x f(X_t) - \E^x f(X_0) - \E^x \int_0^t I(p) f(X_{s-}) ds
\\ & = T_t f(x) - f(x) - \int_0^t \E^x I(p) f(X_{s-}) ds
\\ & = T_t f(x) - f(x) - \int_0^t T_s I(p) f(x) ds
\end{align*}
for every $x\in\R^d$, $t\geq0$. Note that we can switch integrals in line 2 because $I(p)f\in C_b(\R^d)$ by Remark 4.5(ii) in Schilling \cite{SchillingCb}. This implies 
\begin{equation*}
\label{boch1}
\frac{1}{t}(T_t f - f) = \frac{1}{t} \int_0^t T_s I(p) f \hspace{.1cm} ds.
\end{equation*}
 We argue that when taking the limit as $t\searrow0$, the right hand-side converges locally uniformly to $I(p) f$. 
Note that since $I(p)f\in C_b(R^d)$, then $(T_s I(p) f) \mathbbm{1}_K$ is continuous in $s$ for every compact set $K$ by the $C_b$-Feller property, i.e. 
\begin{align*}
||(T_{s+h} I(p) f)\mathbbm{1}_K - (T_{s} I(p) f)\mathbbm{1}_K||_\infty 
 & = \sup_{x\in K} |T_{s+h} I(p) f(x) - T_{s} I(p) f(x)| \rightarrow0
\end{align*}
So, the function $T_{(\LargerCdot)} I(p) f \mathbbm{1}_K$ is the integrand of a Bochner-type integral that is continuous in $s$ and integrable on any closed interval $[a,b]$. Therefore, by Fundamental Theorem of Calculus for Bochner integrals \cite[p.21-22]{Dynkin1965}, 
\begin{align*}
\lim_{t\searrow0}\frac{1}{t}(T_t f - f)\mathbbm{1}_K & = \lim_{t\searrow0} \frac{1}{t} \int_0^t (T_s I(p) f) \mathbbm{1}_K\hspace{.1cm} ds  = (I(p) f)\mathbbm{1}_K
\end{align*}
for all $K\subset\R^d$ compact. Hence, $I(p)f = \lim_{t\searrow0}\frac{1}{t}(T_t f - f)$, where convergence is locally uniform.
\end{proof}

\subsection*{Proof of Lemma \ref{integroderivative}}

\begin{proof}
By Lemma A.\ref{cbfeller}, our semigroup $(T_t)_{t\geq0}$ satisfies the $C_b$-Feller property. Choose $f\in C_b^2(\R^d)$. 
Observe that for all $x\in\R^d$, 
\begin{align*}
T_{t+h}f(x) -T_tf(x) & = T_t (T_hf(x) - f(x)) = T_t \int_0^h T_s I(p) f(x) \,ds
\\ & = \E^x \int_0^h T_s I(p) f(X_t) \,ds =  \int_0^h \E^x T_s I(p) f(X_t) \,ds, \text{ by Fubini's Theorem},
\\ & =  \int_0^h T_t T_s I(p) f(x) \,ds = \int_0^h T_s T_t I(p) f(x)\,ds.
\end{align*}
Thus,
\begin{align*}
\lim_{h\rightarrow0}\frac{1}{h}(T_{t+h}f -T_tf) & = \lim_{h\rightarrow0}\frac{1}{h} \int_0^h T_s T_t I(p) f \hspace{.1cm}ds = T_t I(p) f
\end{align*}
because $T_t I(p) f\in C_b(\R^d)$ by $C_b$-Feller property, thus making $T_s T_t I(p) f \mathbbm{1}_K$  continuous in $s$ for every compact $K$. Once again, by Fundamental Theorem of Calculus for Bochner integrals (see \cite[p.21-22]{Dynkin1965}), we get the convergence shown above.
\par
 Finally, we want to show $I(p) T_t f = T_t I(p) f$.  Choosing $(\phi_n)_{n\in \N}\subset C_c^\infty(\R^d)$ such that $\mathbbm{1}_{B(0,n)} \leq \phi_n \leq 1$ for all $n$. Hence, $f\phi_n \in C_c^2(\R^d) \subset \mcl{D}(\mcl{A})$, the domain of generator $\mcl{A}$, and we have $I(p) T_t f \phi_n = T_t I(p) f\phi_n$. By an approximation argument, we get our desired result.
\end{proof}

\subsection*{Proof of Lemma \ref{extcauchy}}

\begin{proof}
Observe that all limits (and corresponding derivatives) we take here are with respect to locally uniform convergence. Note that by the assumption $x\mapsto p(x,0)$ is continuous, our semigroup $(T_t)_{t\geq0}$ satisfies the $C_b$-Feller property by Lemma A.\ref{cbfeller}. Also, by Lemma \ref{integrogenerates}, we have $\lim_{t\searrow0} \frac{1}{t}(T_t u - u) = I(p)u$ for all $u\in C_b^2(\R^d)$. Observe that we will define the derivative $F'(s)$ by $F'(s) = \lim_{h\rightarrow0} \frac{F(s+h)-F(s)}{h}$ where the limit is under locally uniform convergence. Also, our statement of (b) is different then Liggett's. 
\begin{center}
Liggett's: if $t_n\rightarrow t$, then $||G(t_n) - G(t)||_\infty \rightarrow0$ as $n\rightarrow\infty$.
\vskip.5cm
Ours: if $t_n\rightarrow t$, then $||G(t_n) - G(t)||_{\infty,K} \rightarrow0$ as $n\rightarrow\infty$ for all $K$ compact.
\end{center}
Though Liggett's assumption would be sufficient, we don't need something that strong in our setting, and our $G$ will satisfy locally uniform continuity. Choose a compact set $K\subset\R^d$.
\begin{align*}
& \frac{T_{t-s-h}F(s+h) - T_{t-s}F(s)}{h}  \cdot \mathbbm{1}_K
\\ & = \frac{T_{t-s-h}F(s+h)}{h}   \cdot \mathbbm{1}_K - \frac{T_{t-s}F(s)}{h}  \cdot \mathbbm{1}_K
\\ & + [T_{t-s-h} - T_{t-s}]F'(s)  \cdot \mathbbm{1}_K -  [T_{t-s-h} - T_{t-s}]F'(s)  \cdot \mathbbm{1}_K
\\ & + \frac{T_{t-s-h}F(s)}{h}   \cdot \mathbbm{1}_K- \frac{T_{t-s-h}F(s)}{h}  \cdot \mathbbm{1}_K
\\ & + \frac{T_{t-s}F(s+h)}{h}  \cdot \mathbbm{1}_K - \frac{T_{t-s}F(s+h)}{h} \cdot \mathbbm{1}_K
\\ & + \frac{T_{t-s}F(s)}{h}  \cdot \mathbbm{1}_K - \frac{T_{t-s}F(s)}{h}  \cdot \mathbbm{1}_K
\\ & =: (1) + (2) + (3) + (4) + (5) + (6) + (7) + (8) + (9) + (10)
\\ & = \textcolor{red}{[(2)+(7)]} + \textcolor{blue}{[(5)+(10)]} + \textcolor{ForestGreen}{[(3)]} + \textcolor{Plum}{[ (4) + (1) + (9) + (8) + (6)]}
\\ & = \textcolor{red}{T_{t-s}\left[ \frac{F(s+h) - F(s)}{h}\right]  \cdot \mathbbm{1}_K }+ \textcolor{blue}{\left[ \frac{T_{t-s-h}-T_{t-s}}{h} \right]F(s)  \cdot \mathbbm{1}_K}
\\ & + \textcolor{ForestGreen}{[T_{t-s-h} - T_{t-s}]F'(s)  \cdot \mathbbm{1}_K}  + \textcolor{Plum}{[T_{t-s-h}-T_{t-s}]\left[ \frac{F(s+h) - F(s)}{h} - F'(s)\right]  \cdot \mathbbm{1}_K}
\\ & =: \textcolor{red}{(I)} + \textcolor{blue}{(II)} + \textcolor{ForestGreen}{(III)} + \textcolor{Plum}{(IV)}
\end{align*}
Now we consider the limits as $h$ goes to $0$ for each of these four terms.
\begin{align*}
(I): \lim_{h\searrow0} T_{t-s} \left[\frac{F(s+h)- F(s)}{h} \right]  \cdot \mathbbm{1}_K& = T_{t-s} \lim_{h\searrow0} \left[\frac{F(s+h)- F(s)}{h} \right]  \cdot \mathbbm{1}_K  = T_{t-s} F'(s)  \cdot \mathbbm{1}_K
\end{align*}
because $T_{t-s}$ is a bounded operator, which means it is a continuous operator. \mypar 

\noindent (II): Let $u = t-s$. Then $s = t-u$ and $ds = -du$. For a function $f\in C_b(\R^d)$, 
\begin{align*}
\lim_{h\searrow0} \left[ \frac{T_{t-s-h}-T_{t-s}}{h} \right]f  \cdot \mathbbm{1}_K = \frac{d}{ds} T_{t-s} f  \cdot \mathbbm{1}_K = -\frac{d}{du} T_{u} f  \cdot \mathbbm{1}_K &= -I(p)T_{u} f  \cdot \mathbbm{1}_K \\ & = -I(p) T_{t-s} f  \cdot \mathbbm{1}_K.
\end{align*}
Therefore, $\displaystyle\lim_{h\searrow0} \left[ \frac{T_{t-s-h}-T_{t-s}}{h} \right]F(s)  \cdot \mathbbm{1}_K=  -I(p) T_{t-s} F(s) \cdot \mathbbm{1}_K = - T_{t-s} I(p) F(s)  \cdot \mathbbm{1}_K.$
\\ \\
(III):
By $C_b$-Feller property, since $F'(s)\in C_b(\R^d)$, $\displaystyle\lim_{h\searrow0} [T_{t-s-h} -T_{t-s}]F'(s)  \cdot \mathbbm{1}_K  = 0$ uniformly.
\\ \\
(IV): Observe that $T_{t-s-h}$ and $T_{t-s}$ are both contractions. Hence, 
\begin{align*}
& \left| \left| [T_{t-s-h}-T_{t-s}]\left[ \frac{F(s+h) - F(s)}{h} - F'(s)\right] \right| \right|_{\infty,K}
\\ & \leq \left| \left| T_{t-s-h}-T_{t-s} \right| \right|\cdot  \left| \left| \left[ \frac{F(s+h) - F(s)}{h} - F'(s)\right] \right| \right|_{\infty,K}
\\ & \leq 2\left| \left| \left[ \frac{F(s+h) - F(s)}{h} - F'(s)\right] \right| \right|_{\infty,K}\longrightarrow0
\end{align*}
as $h\rightarrow0$. Thus, we have for $0<s<t$,
\begin{align*}
\frac{d}{ds} & T_{t-s} F(s)  \cdot \mathbbm{1}_K  = \lim_{h\searrow0} \frac{T_{t-(s+h)} F(s+h) - T_{t-s}F(s)}{h}  \cdot \mathbbm{1}_K
\\ & = \lim_{h\searrow0} [(I)+(II)+(III)+(IV)] = T_{t-s} F'(s)  \cdot \mathbbm{1}_K - T_{t-s} I(p) F(s) \cdot \mathbbm{1}_K
\\ & = T_{t-s}[F'(s) - I(p) F(s)]  \cdot \mathbbm{1}_K \stackrel{(c)} = T_{t-s} G(s)  \cdot \mathbbm{1}_K.
\end{align*}
The right-hand side is a continuous function of $s$ because $G$ is continuous function of $s$ and the semigroup is uniformly continuous on $K$ by $C_b$-Feller property.  Let's justify this:
\\ \\
\underline{Aside}: Let $\epsilon>0$. Then $\exists N$ large s.t. $||G(s_n) - G(s)||_{\infty,K}<\epsilon/2$ for all $n\geq N$. Also, $\exists N'$ large s.t. 
$||T_{t-{s_n}} G(s_{N}) - T_{t-s} G(s_{N})||_{\infty,K}  = ||(T_{t-s_n}-T_{t-s}) G(s_{N}) ||_{\infty,K} <\epsilon/2$
for all $n\geq N'$ since semigroup operator is uniformly continuous on compact sets. Let $M=\max(N,N')$.  
\begin{align*}
||T_{t-s_M} G(s_M) - T_{t-s} G(s)||_{\infty,K} & = ||T_{t-s_M} G(s_M) - T_{t-s} G(s_M) +T_{t-s} G(s_M) - T_{t-s} G(s)||_{\infty,K}
\\ & \leq ||T_{t-s_M} G(s_M) - T_{t-s} G(s_M)||_{\infty,K} + ||T_{t-s} G(s_M) - T_{t-s} G(s)||_{\infty,K}
\\ & \leq ||T_{t-s_M} G(s_M) - T_{t-s} G(s_M)||_{\infty,K} + || G(s_M) - G(s)||_{\infty,K}
\\ & < \epsilon/2 + \epsilon/2 = \epsilon.
\end{align*}

Therefore we can integrate these functions with respect to $s$ from $0$ to $t$. And by Fundamental Theorem of Calculus for Bochner integrals (see \cite[p.21-22]{Dynkin1965}),
\begin{align*}
\int_0^t T_{t-s} G(s) ds \hspace{-.05cm}  \cdot \hspace{-.075cm} \mathbbm{1}_K  = \int_0^t \frac{d}{ds} T_{t-s} F(s) ds   \hspace{-.05cm}  \cdot \hspace{-.075cm} \mathbbm{1}_K  = (T_{t-t}F(t) - T_t F(0))   \mathbbm{1}_K =(F(t)  - T_t F(0))  \mathbbm{1}_K.
\end{align*}

\noindent Since $K$ compact is  arbitrary, we have our desired result: $F(t) = T_t F(0) + \int_0^t T_{t-s} G(s)ds$. 
\end{proof}

\newpage

\bibliographystyle{acm}

\bibliography{references-paper1_tu}

\end{document}